\documentclass[12pt]{article}

\usepackage{amssymb}%
\usepackage{amsmath}%
\usepackage{amsthm}%

\usepackage{xcolor}%

\allowdisplaybreaks%

{\theoremstyle{plain}%
 \newtheorem{theorem}{Theorem}

}
{\theoremstyle{remark}

}
{\theoremstyle{definition}

}

\begin{document}

\begin{center}
{\Large An extension of the Chudnovsky algorithm}
\end{center}

\begin{center}
{\textsc{John M. Campbell}} 

 \ 

\end{center}

\begin{abstract}
Using an infinite family of generalizations of the Chudnovsky brothers' series recently obtained via the analytic continuation of the Borwein brothers' formula 
 for Ramanujan-type series of level 1, we apply the Gauss--Salamin--Brent iteration for $\pi$ to obtain a new, Ramanujan-type series that yields more 
 digits per term relative to current world record given by an extension of the Chudnovsky algorithm from Bagis and Glasser that produces about 110 digits per 
 term. We explicitly evaluate the required nested radicals over $ \mathbb{Q}$ involved in the our  summation, which yields about 153 digits per 
 term, and we provide a practical way of implementing our higher-order version of the Chudnovsky algorithm via the the PARI/GP system. An evaluation 
 due to Berndt and Chan for the modular $j$-invariant associated with their order-3315 extension of the Chudnovskys' Ramanujan-type series provides 
 a key to our applications of recursions for the elliptic lambda and elliptic alpha functions. 
\end{abstract}



\section{Introduction}
 Famously, the Chudnovsky brothers applied their formula 
\begin{align}
\begin{split}
 & \frac{711822400}{\sqrt{10005}\, \pi } = \\ 
 & \sum_{n=0}^{\infty} \frac{ \left( \frac{1}{6} \right)_{n} \left( \frac{1}{2} \right)_{n} \left( \frac{5}{6} \right)_{n} }{ \left( 1 \right)_n^{3} }
 \left( -\frac{1}{53360} \right)^{3n}
 \left( \frac{13591409}{6} + 90856689 n \right), 
\end{split}\label{Chudnovskymain}
\end{align}
 which they gave in \cite{ChudnovskyChudnovsky1988}, to compute over 8 billion base-10 digits of $\pi$ \cite[p.\ 111]{ArndtHaenel2001}, with the 
 \emph{Chudnovsky algorithm} being based on the Ramanujan-type series given above. The Chudnovsky algorithm is of a seminal nature both in terms 
 of the numerical computation of $\pi$ and within number-theoretic areas concerning modular relations associated with Ramanujan-type series. The 
 Chudnovsky brothers' Ramanujan-type series in \eqref{Chudnovskymain} has been applied in practical ways, and a notable application of 
 \eqref{Chudnovskymain} is given by its implementation in Computer Algebra System (CAS) software such as Mathematica to calculate the digits of $\pi$ 
 \cite[p.\ 111]{ArndtHaenel2001}. This leads us to develop techniques that may be applied to obtain higher-order versions of the Chudnovsky algorithm 
 that provide more digits per term, relative to the approximately 14 decimal digits per term given by the Chudnovskys' series in \eqref{Chudnovskymain}. 
 In this regard, a notable development in the high-precision computation of $\pi$ via Ramanujan-type series is due to Berndt and Chan 
 \cite{BerndtChan2001}, who introduced a higher-order analogue of \eqref{Chudnovskymain} that produces about 74 digits per term. This provided a 
 world record for the series for $\frac{1}{\pi}$ producing the most digits per term \cite{BerndtYee2002}. Shortly afterwards, Takeshi Sato set the world 
 record, in an unpublished communication, for a series for $\frac{1}{\pi}$ yielding about 97 digits per term, with reference to the survey 
 \cite{BaruahBerndtChan2009} by Baruah, Berndt, and Chan on Ramanujan's series for $\frac{1}{\pi}$. In a subsequent paper due to Bagis and Glasser 
 \cite{BagisGlasser2012}, the authors used a special case of a family of Ramanujan-type series due to Berndt and Chan \cite{BerndtChan2001} to 
 produce a Ramanujan-type series yielding 110 digits per term. In this article, we apply a recursive technique derived from the Gauss--Salamin--Brent 
 iteration for $\pi$ \cite[p.\ 363]{BorweinBorwein1988} together with an infinite family of generalizations of \eqref{Chudnovskymain} recently introduced in 
 \cite{Campbell2024}, to obtain an extension of \eqref{Chudnovskymain} and of Berndt and Chan's series and of Bagis and Glasser's series, yielding 
 more digits per term relative to the current world record. 

 The Ramanujan-type series we introduce produces about 153 digits per term and may be implemented in a practical way using the PARI/GP system. 
 A remarkable aspect about the rapidly converging sum we introduce is given by how neither Mathematica nor Maple is able to produce sufficiently 
 high-precision calculations required for the nested radicals over $\mathbb{Q}$ involved in our work. This is in contrast to how Bagis and Glasser relied 
 on Mathematica in \cite{BagisGlasser2012}. 

 As in \cite{BerndtChan2001,BerndtYee2002}, we set 
 $$ \text{{\bf J}}(q) := \frac{1728}{j\left( \frac{3 + \tau}{2} \right)}, $$ 
 with $\text{{\bf J}}_{n} = \text{{\bf J}}\big( e^{-\pi \sqrt{n}} \big)$ for $n > 0$, 
 where the modular $j$-invariant 
 is such that 
 $$ j_{n} := j\left( \frac{3 + \sqrt{-n}}{2} \right) = 1728 \frac{Q_{n}^{3}}{Q_{n}^{3} - R_{n}^{2}}, $$
 writing $Q_{n} := Q\big( -e^{-\pi \sqrt{n}} \big)$ and $R_{n} := R\big( -e^{-\pi \sqrt{n}} \big)$ 
 to denote Eisenstein series of the forms indicated below, 
 with $q = \text{exp}(2 \pi i \tau)$, referring to \cite{BerndtChan2001,BerndtYee2002} for details: 
\begin{align*}
 Q(q) & := 1 + 240 \sum_{k=1}^{\infty} \frac{k^3 q^{k}}{1 - q^{k}}, \\ 
 R(q) & := 1 - 504 \sum_{k=1}^{\infty} \frac{ k^5 q^{k} }{1 - q^{k}}. 
\end{align*}
 The evaluation for $ \text{{\bf J}}_{3315} $ due to Berndt and Chan \cite{BerndtChan2001} is of central importance, for our purposes, and this 
 evaluation is reproduced in \eqref{boldJclosed} below and is given in terms of a closed form for $\lambda_{1105}$, where the function 
 $\lambda_{n}$ was introduced by Ramanujan and defined via a quotient of Dedekind eta-functions, referring to \cite{BerndtChanKangZhang2002} 
 for details. Our method relies on recursions for the elliptic lambda and elliptic alpha functions and builds upon Berndt and Chan's construction based on 
 special values for Eisenstein series and the modular $j$-invariant. 

\subsection{Background and preliminaries}
 In addition to the survey \cite{BaruahBerndtChan2009} referenced above, 
 references as in 
 \cite{ChanCooper2012,ChenGlebovGoenka2022,Chu2011,Guillera2013,HanChen2022,HuberSchultzYe2023,Liu2012,Wan2014,WangYang2022,WeiGong2013} 
 concerning Ramanujan-type series for $\frac{1}{\pi}$ have inspired our construction of a new Ramanujan-type series extending
 the Chudnovsky brothers' series. 
 For background material on 
 the importance of Ramanujan-type series within computational mathematics, number theory, and many other areas, 
 the interested reader is invited to review the provided references and references therein. 
 For the the sake of brevity, we proceed to cover required preliminaries, as below, 
 derived from the \emph{Pi and the AGM} text \cite{BorweinBorwein1987} 
 and related to Berndt and Chan's record-setting series for $\frac{1}{\pi}$ \cite{BerndtChan2001,BerndtYee2002}. 

 The \emph{complete elliptic integral of the first kind} is such that 
\begin{equation}\label{definitionK}
 \text{{\bf K}}(k) := \int_{0}^{\frac{\pi}{2}} \left( 1 - k^2 \sin^2\theta \right)^{-\frac{1}{2}} \, d\theta, 
\end{equation}
 and the expression $k$ in \eqref{definitionK} is referred to as the \emph{modulus}. 
 The \emph{complementary modulus} is such that $k' := \sqrt{1 - k^2}$, 
 and we write $\text{{\bf K}}'(k) := \text{{\bf K}}(k')$. 
 The \emph{elliptic lambda function} $\lambda^{\ast}$ \cite[\S3.2]{BorweinBorwein1987} is such that 
\begin{equation}\label{ellipticlambdadefinition}
 \frac{ \text{{\bf K}}'\left( \lambda^{\ast}(r) \right) }{ 
 \text{{\bf K}}\left( \lambda^{\ast}(r) \right) } = \sqrt{r} 
\end{equation}
 for $r > 0$. The \emph{complete elliptic integral of the second kind} is such that 
\begin{equation}\label{completeEdefinition}
 \text{{\bf E}}(k) := \int_{0}^{\frac{\pi}{2}} \left( 1 - k^2 \sin^2 \theta \right)^{\frac{1}{2}} \, d\theta, 
\end{equation}
 and the special function in \eqref{completeEdefinition}
 may be used to define the \emph{elliptic alpha function} \cite[\S5]{BorweinBorwein1987} such that 
\begin{equation}\label{alphadefinition}
 \alpha(r) = \frac{\pi}{4 \text{{\bf K}}^{2}\left( \lambda^{\ast}(r) 
 \right)} - \sqrt{r} \left( \frac{\text{{\bf E}}\left( \lambda^{\ast}(r) \right)}{ \text{{\bf K}}\left( 
 \lambda^{\ast}(r) \right) } - 1 \right). 
\end{equation}
 \emph{Ramanujan-type series} are such that 
\begin{equation}\label{Rseriesdefinition}
 \frac{1}{\pi} = \sum_{n=0}^{\infty} 
 \frac{ \left( \frac{1}{2} \right)_{n} \left( \frac{1}{s} \right)_{n} \left( 1 - \frac{1}{s} \right)_{n} }{ 
 \left( 1 \right)_{n}^{3} } z^{n} (a + b n) 
\end{equation}
 for $s \in \{ 2, 3, 4, 6 \}$ and for real and algebraic values $a$, $b$, and $z$. The \emph{level} $\ell$ of \eqref{Rseriesdefinition} may defined so that 
 $\ell:=4\sin^2\frac{\pi}{s}$ and refers to the level of modular forms underlying the classical derivations of Ramanujan-type series. Our recent work 
 concerning the Heegner number 163 and its connection to the Chudnovsky algorithm \cite{Campbell2024} included an infinite family of generalizations of 
 \eqref{Chudnovskymain} in terms of the special functions defined in \eqref{ellipticlambdadefinition} and \eqref{alphadefinition}. 

 An infinite family of Ramanujan-type series for the $s = 6$ case of \eqref{Rseriesdefinition} was introduced in the classic \emph{Pi and the AGM} text 
 \cite[\S5.5]{BorweinBorwein1987}. However, this infinite family strictly applies to the case whereby $z > 0$ in \eqref{Rseriesdefinition}, in contrast to the 
 Chudnovskys' series in \eqref{Chudnovskymain}. The Borwein brothers indicated that their infinite family of level $\ell = 1$ Ramanujan-type 
 series could be extended, with the use of analytic continuation, in order to allow for negative values of $z$. 
 This extension was given recently in \cite{Campbell2024} and is reproduced below and 
 was applied in \cite{Campbell2024} to prove a conjecture due to Bagis and Glasser 
 on a Ramanujan-type series related to the Heegner number 163 \cite{BagisGlasser2013}. 

 Throughout this paper, we set 
\begin{equation}\label{xrdefinition}
 x = x(r) = 4 \left( \left( \lambda^{\ast}(r) \right)^{2} - \left( \lambda^{\ast}(r) \right)^{4} \right). 
\end{equation}
 We set $s = 6$, and we let $r > 1$. The formula in \eqref{Rseriesdefinition} then holds for the following values \cite{Campbell2024}: 
\begin{align}
 & z = -\frac{27 x}{(1 - 4 x)^3}, \ \ \ b = \frac{(8x+1) \sqrt{1-x} \sqrt{r} }{ \left( 1 - 4 x \right)^{3/2} }, \\ 
 & a = \frac{2(1 - 4 x) \alpha(r) + (4 x - 1 + \sqrt{1-x}) \sqrt{r} }{2 \left( 1 - 4 x \right)^{3/2}}. \label{acoefficient}
\end{align}

\section{Main consturction}\label{sectionMain}
 Our extension of the Chudnovsky algorithm requires the special values for $\lambda^{\ast}$ and $\alpha$ highlighted in Theorem \ref{maintheorem} 
 below. Our construction based on these closed forms requires the closed forms for the expressions $t_{3315}$ and $\text{{\bf J}}_{3315}$ introduced 
 in \cite{BerndtChan2001} and indicated below. 

 Berndt and Chan \cite{BerndtChan2001} introduced and proved the following formula
 for an infinite family of Ramanujan-type series, referring to \cite{BerndtChan2001,BerndtYee2002} for details: 
\begin{equation}\label{BCfamily}
 \frac{6}{\sqrt{m}\sqrt{1 - \text{{\bf J}}_{m}}} \frac{1}{\pi} = \sum_{n = 0}^{\infty} 
 \frac{ \left( \frac{1}{6} \right)_{n} \left( \frac{5}{6} \right)_{n} 
 \left( \frac{1}{2} \right)_{n} }{ \left( 1 \right)_{n}^{3} } \text{{\bf J}}_{m}^{n} 
 (6n + 1 - t_{m}). 
\end{equation}
 The following explicit evaluation for $t_{3315}$ was proved in \cite{BerndtChan2001}: 
\begin{align*}
 t_{3315} = &
 \frac{1095255033002752301233099478037584}{2050242335692983321671746996556833} + \\ 
 & \frac{1006588064225996719872149534306400}{34854119706780716468419698941466161} \sqrt{5} \sqrt{17} + \\
 & \frac{692779168175128551453280427070000}{34854119706780716468419698941466161} \sqrt{17} - \\
 & \frac{136434536163779492503565618457696}{2050242335692983321671746996556833} \sqrt{5} + \\
 & \frac{400179322879781860521299209248000}{26653150364008783181732710955238829} \sqrt{13} + \\
 & \frac{1077564413015882021519209726762688}{453103556188149314089456086239060093} \sqrt{5} \sqrt{13} \sqrt{17} + \\
 & \frac{120226784218523863048087030809600}{64729079455449902012779440891294299} \sqrt{13} \sqrt{17} + \\
 & \frac{239369594240980944219359445009600}{26653150364008783181732710955238829} \sqrt{5} \sqrt{13}. 
\end{align*}
 The following closed form for $\lambda_{1105}$ from \cite{BerndtChan2001} 
 allows us to obtain the value for $\text{{\bf J}}_{3315}$ required for the $m = 3315$
 case of \eqref{BCfamily}: 
 $$ \lambda_{1105} = \frac{\left(1+\sqrt{5}\right)^{12} \left(4+\sqrt{17}\right)^3 \left(8+\sqrt{65}\right)^3 \left(15+\sqrt{221}\right)^3}{32768}. $$ 
 Explicitly, we have that: 
\begin{equation}\label{boldJclosed}
 \text{{\bf J}}_{3315} 
 = -\frac{64 \lambda_{1105}^{2}}{ \left( \lambda_{1105}^{2} - 1 \right) \left( 9 \lambda_{1105}^{2} - 1 \right)^{3} }. 
\end{equation}
 The $m = 3315$ case of \eqref{BCfamily}
 gives us, in an equivalent way 
 via the application of Clausen's hypergeometric product formula in \cite{BerndtChan2001}, the $r = 3315$ case of 
 \eqref{Rseriesdefinition} for the specified values among 
 \eqref{xrdefinition}--\eqref{acoefficient}, with $s = 6$. 
 
 As suggested above, a remarkable aspect of Theorem \ref{maintheorem} is given by how neither Mathematica nor Maple 
 is not able to compute required values as in \eqref{boldx3315}, 
 which require higher-precision numerical estimates made available via the PARI/GP software. 

\begin{theorem}\label{maintheorem}
 Let 
\begin{equation}\label{boldx3315}
 \text{{\bf x}}_{3315} = \frac{1}{4} - 
 \frac{3}{8} \frac{ \left(-\frac{1}{ \text{{\bf J}}_{3315} } \right)^{\frac{2}{3}}}{ 
 \left( \sqrt{ \frac{ \text{{\bf J}}_{3315} - 1 }{\text{{\bf J}}_{3315}} } - 1 \right)^{\frac{1}{3}} } 
 + \frac{3}{8} \left( \frac{1 - \sqrt{ \frac{ \text{{\bf J}}_{3315} - 1}{\text{{\bf J}}_{3315}} }}{\text{{\bf J}}_{3315}} 
 \right)^{\frac{1}{3}}. 
\end{equation}
 Then 
\begin{equation}\label{eqlambdastar3315}
 \lambda^{\ast}(3315) = \frac{\sqrt{1 - \sqrt{1 - \text{{\bf x}}_{3315}}}}{\sqrt{2}} 
\end{equation}
 and 
\begin{equation}\label{eqlambdastar13260}
 \lambda^{\ast}(13260) = \frac{1 - \sqrt{1 - \left( \lambda^{\ast}(3315) \right)^{2} }}{1 + 
 \sqrt{1 - \left( \lambda^{\ast}(3315) \right)^{2} }}, 
\end{equation}
 and, for the same value of $\text{{\bf x}}_{3315}$ in \eqref{boldx3315}, 
 the elliptic alpha value $ \alpha(3315) $ equals 
$$ \text{{\footnotesize $\frac{ \frac{1}{2} \sqrt{\frac{1105}{3}} \sqrt{1 - \text{{\bf J}}_{3315}} (1 - t_{3315}) 2 (1 - 4 \text{{\bf x}}_{3315})^{\frac{3}{2}} - 
 (4 \text{{\bf x}}_{3315} - 1 + \sqrt{1 - \text{{\bf x}}_{3315} }) \sqrt{3315}}{2 (1 - 4 \text{{\bf x}}_{3315})},$}} $$
 with 
\begin{equation}\label{eqalphaat13260}
 \alpha(13260) = 
 \frac{4 \alpha(3315) - 2 \sqrt{3315} \left( \lambda^{\ast}(3315) \right)^2}{ \left(1 + \sqrt{1 - 
 \left( \lambda^{\ast}(3315) \right)^{2} } \right)^2}. 
\end{equation}
\end{theorem}

\begin{proof}
 We normalize the $m = 3315$ case of \eqref{BCfamily}, writing 
\begin{align*}
 \frac{1}{\pi} = \sum_{n = 0}^{\infty} 
 \frac{ \left( \frac{1}{6} \right)_{n} \left( \frac{5}{6} \right)_{n} 
 \left( \frac{1}{2} \right)_{n} }{ \left( 1 \right)_{n}^{3} } \text{{\bf J}}_{3315}^{n} 
 \left( n + \frac{1 - t_{3315}}{6} \right) \sqrt{3315}\sqrt{1 - \text{{\bf J}}_{3315}}, 
\end{align*}
 giving an equivalent formulation of the $r = 3315$ case of \eqref{Rseriesdefinition}, again for $s = 6$ and for the specified values listed among 
 \eqref{xrdefinition}--\eqref{acoefficient}. This equivalence gives us that 
 \begin{equation}\label{J3315intermsofx}
 \text{{\bf J}}_{3315} = -\frac{27 x(3315)}{(1 - 4 x(3315))^3}. 
\end{equation}
 Solving for $x = x(3315)$ according to the relation in \eqref{J3315intermsofx}, we obtain the right-hand expansion in \eqref{boldx3315}, writing 
 $ \text{{\bf x}}_{3315} = x(3315)$. According to the relation between $x(r)$ and $\lambda^{\ast}(r)$ in \eqref{xrdefinition}, we solve for 
 $ \lambda^{\ast}(3315)$, producing the right-hand value in \eqref{eqlambdastar3315}. 

 The recursions whereby 
\begin{equation}\label{GSBalpha}
 \alpha(4 N) = \frac{ 4 \alpha(N) - 2 \sqrt{N} \left( \lambda^{\ast}(N) \right)^{2} }{ 
 \left( 1 + \sqrt{1 - \left( \lambda^{\ast}(N) \right)^{2}} \right)^{2} } 
\end{equation}
 and 
\begin{equation}\label{GSBlambda}
 \lambda^{\ast}(4N) = \frac{1 - \sqrt{1 - \left( \lambda^{\ast}(N) \right)^{2}}}{1 + 
 \sqrt{1 - \left( \lambda^{\ast}(N) \right)^{2}}} 
\end{equation}
 are known to be equivalent to the Gauss--Salamin--Brent iteration for $\pi$ \cite[p.\ 363]{BorweinBorwein1988} 
 and are proved in the \emph{Pi and the AGM} text \cite[pp.\ 158--159]{BorweinBorwein1987}. 
 The latter recursion in \eqref{GSBlambda} gives us the formula for
 $\lambda^{\ast}(13260)$ in \eqref{eqlambdastar13260}. 
 Again for $r = 3315$, we set the right-hand side of \eqref{acoefficient}
 to be equal to 
 $$ \frac{1-t_{3315}}{6} \sqrt{3315} \sqrt{1 - \text{{\bf J}}_{3315}}, $$ 
 and, by solving for $\alpha(3315)$, we obtain the given expression of $\alpha(3315)$
 involving $\text{{\bf x}}_{3315}$ and $\text{{\bf J}}_{3315}$. 
 The recursion in \eqref{GSBalpha} then gives us the 
 formula for $ \alpha(13260) $ in \eqref{eqalphaat13260}. 
\end{proof}

 Theorem \ref{maintheorem} allows us to express $\lambda^{\ast}(13260)$ and $\alpha(13260)$ explicitly via nested radicals over $\mathbb{Q}$, 
 and, as illustrated in Section \ref{sectionNumerical} below, the PARI/GP system provides an efficient way of numerically computing these required values. 
 By again setting $s = 6$ in \eqref{Rseriesdefinition}, and by again enforcing the relations among 
 \eqref{xrdefinition}--\eqref{acoefficient}, by then setting $r = 13260$, Theorem \ref{maintheorem} 
 allows us to explicitly evaluate the $a$- and $b$- and $z$-values 
 within the summand in \eqref{Rseriesdefinition}. 
 This gives us a new Ramanujan-type series that may be seen
 as accelerating the convergence of the Chudnovsky brothers' 
 series, the Berndt--Chan series, and Bagis and Glasser's series. 

\section{Numerical analysis}\label{sectionNumerical}
 For a hypergeometric series converging to $L$ and with an absolute convergence rate of $\mu$, the number of digits per term is, informally, about 
 $ \text{log}_{10}\big( \frac{1}{\mu} \big)$. For example, by setting $\mu$ as the absolute convergence rate of the Chudnovsky's series in 
 \eqref{Chudnovskymain}, we find that $\text{log}_{10}\big( \frac{1}{\mu} \big) \approx 14.1816$, 
 and this agrees with how it is well known that the Chudnovsky algorithm produces about 14 digits per term. 
 Similarly, by computing $\text{log}_{10}\big( -\frac{1}{\text{{\bf J}}_{3315}} \big) \approx 75.3178$, 
 this agrees with how it has been stated that the Berndt--Chan series produces about 74 digits per term. 
 The PARI/GP input below (where display breaks are permitted) provides a practical and efficient way of implementing 
 the Ramanujan-type series introduced in this paper, 
 and allows us to numerically compute its rate of convergence, as below. 
 Remarkably, Mathematica and Maple are unable to 
 provide sufficiently precise numerical evaluations 
 related to the required algebraic expressions over $\mathbb{Q}$ involving nested radicals. 
{\footnotesize
\begin{verbatim}
default(realprecision, 1000); lambda1105=((sqrt(5) + 1)/2)^12*(4 + 
sqrt(17))^3*((15 + sqrt(221))/2)^3*(8 + sqrt(65))^3;
boldJ3315=-((64*lambda1105^2)/((lambda1105^2 - 1)*(9*lambda1105^2 - 
1)^3));
x3315=1/8*(2 - (3*(-(1/boldJ3315))^(2/3))/(-1 + sqrt((-1 + 
boldJ3315)/boldJ3315))^(1/3) + 3*((1 - sqrt((-1 + 
boldJ3315)/boldJ3315))/boldJ3315)^(1/3));
lambdaast3315=sqrt(1-sqrt(1-x3315))/sqrt(2);
lambdaast13260=(1-sqrt(1-lambdaast3315^2))/(1+sqrt(1-
lambdaast3315^2));
x13260=4*(lambdaast13260^2-lambdaast13260^4);
z13260=-27*x13260/(1-4*x13260)^3;
t3315 = 1095255033002752301233099478037584/
2050242335692983321671746996556833 + 
1006588064225996719872149534306400/
34854119706780716468419698941466161*sqrt(17)*sqrt(5) + 
692779168175128551453280427070000/
34854119706780716468419698941466161*sqrt(17) - 
136434536163779492503565618457696/
2050242335692983321671746996556833*sqrt(5) + 
400179322879781860521299209248000/
26653150364008783181732710955238829*sqrt(13) +
1077564413015882021519209726762688/
453103556188149314089456086239060093*sqrt(13)*sqrt(17)*sqrt(5) +
120226784218523863048087030809600/
64729079455449902012779440891294299*sqrt(17)*sqrt(13) +
239369594240980944219359445009600/
26653150364008783181732710955238829*sqrt(13)*sqrt(5);
alpha3315 = (1/2*sqrt(1105/3)*sqrt(1 - boldJ3315)*(1 - 
t3315)*2*(1 - 4*x3315)^(3/2) - (4*x3315 - 1 + sqrt(1 - 
x3315))*sqrt(3315))/(2*(1 - 4*x3315));
alpha13260 = (4*alpha3315 - 2*sqrt(3315)*lambdaast3315^2)/(1 + 
sqrt(1 - lambdaast3315^2))^2;
b13260 = ((8*x13260 + 1)*sqrt(1 - x13260)*sqrt(13260))/(1 - 
4*x13260)^(3/2);
a13260 = (2*(1 - 4*x13260)*alpha13260 + (4*x13260 - 1 + sqrt(1 - 
x13260))*sqrt(13260))/(2*(1 - 4*x13260)^(3/2));
\end{verbatim}}

 By inputting 
{\footnotesize
\begin{verbatim}
log(-1/z13260)/log(10)
\end{verbatim}}
\noindent into PARI/GP after the preceding input, we obtain a value of about 153.8732. We would thus expect the Ramanujan-type series introduced in Section 
 \ref{sectionMain} to provide about 153 digits per term. This turns out to be an accurate estimate, as illustrated below. 

 Inputting 
{\footnotesize
\begin{verbatim}
sum(n=0,0,prod(i=0,n-1,1/2+i)*prod(i=0,n-1,1/6+i)*prod(i=0,n-1,5/6+
i)/(prod(i=0,n-1,1+i))^3*z13260^n*(a13260+b13260*n))
\end{verbatim}}
\noindent into PARI/GP, with the above indicated partial sum consisting of the initial term, if we then subtract $\frac{1}{\pi}$, thie resultant difference is about 
 $1.073 \times 10^{-153}$, which agrees with our new Ramanujan-type series producing about 153 digits per term. Explicitly, we have that the above 
 input produces a numerical evaluation beginning with the following, and we may verify that this agrees with $\frac{1}{\pi}$ for 152 decimal digits after the 
 decimal point, with the first disagreement occurring with the last $1$-digit shown below. 
{\footnotesize
\begin{verbatim}
0.31830988618379067153776752674502872406891929148091289749533468811
7793595268453070180227605532506171912145685453515916073785823692229
157305755934821463401
\end{verbatim}}
 By then computing the decimal expansion of the sum of the first two terms of the new Ramanujan-type series constructed in this paper, this decimal 
 expansion begins as follows, and we may verify that it agrees with the decimal expansion of $\frac{1}{\pi}$ up to and including the last pair of the form 
 {\tt 99} shown below. This gives us that the sum of the first two terms of our extension of the Chudnovsky algorithm provide us with an accuracy of 
 $152 + 154 $ digits after the decimal point. 
{\footnotesize
\begin{verbatim}
0.31830988618379067153776752674502872406891929148091289749533468811
7793595268453070180227605532506171912145685453515916073785823692229
1573057559348214633996784584799338748181551461554927938506153774347
8579243479532338672478048344725802366476022844539951143188092378017
380534791224097882187387568817105744619977773
\end{verbatim}}
 Further evaluations for the partial sums of the Ramanujan-type series we have introduced yield similar results, in view of the expected number of digits per 
 term produced. In addition to how this new series extends the Chudnovsky brothers' Ramanujan-type series and Berndt and Chan's record-breaking 
 series for $\frac{1}{\pi}$, this is also of interest in terms of 
 the relation to the work of the Borwein brothers \cite{BorweinBorwein1993} 
 concerning Ramanujan-type series of class three, 
 which introduced a Ramanujan-type series that produces about 50 digits per term, 
 which provided a world record \cite{BerndtChanKangZhang2002} 
 prior to the world record set by Berndt and Chan in 2001 \cite{BerndtChan2001}. 
 We leave it to a separate project to further extend the Chudnovsky algorithm and Berndt and Chan's Ramanujan-type series. 

\subsection*{Acknowledgements}
The author was supported by a Killam Postdoctoral Fellowship from the Killam Trusts.

\bibliographystyle{amsplain}

 Department of Mathematics and Statistics 

 Dalhousie University 

 Halifax, Nova Scotia B3H 4R2 

 {\tt jmaxwellcampbell@gmail.com}

\end{document}